\documentclass[12pt,a4paper,final]{iopart}

\usepackage[normalem]{ulem}
\usepackage{cite}

\usepackage{iopams}  
\usepackage{graphicx}
\usepackage[breaklinks=true,colorlinks=true,linkcolor=blue,urlcolor=blue,citecolor=blue]{hyperref}
\newcommand{\thetamin}{\theta^{(0)}}
\newcommand{\halflap}{\left(-\frac{\mathrm{d}^2}{\mathrm{d}x^2}\right)^{1/2}}

\newcommand{\brackets}[1]{\left( #1 \right)}

\newcommand{\R}{\mathbb{R}}
\newcommand{\norm}[1]{\left\| #1  \right\|} 
\renewcommand{\d}[1]{\,  \mathrm{d} #1} 
\newcommand{\abs}[1]{\left| #1 \right|} 
\renewcommand{\v}[1]{\ensuremath{\mathbf{#1}}} 
\newcommand{\half}[0]{\frac{1}{2}}
\newcommand{\uv}[1]{\ensuremath{\hat{\mathbf{#1}}}}
\newcommand{\beq}[0]{\begin{equation}}
\newcommand{\eeq}[0]{\end{equation}}
\newcommand{\pd}[2]{\frac{\partial #1}{\partial #2}} 

\newcommand{\weak}{\rightharpoonup}

\newcommand{\grad}[0]{\nabla}

\renewcommand{\div}[0]{\nabla\cdot}
\expandafter\let\csname equation*\endcsname\relax

\expandafter\let\csname endequation*\endcsname\relax

\usepackage{amsmath}
\usepackage{amsthm}
\newtheorem{thm}{Theorem}
\newtheorem{lemma}{Lemma}
\newtheorem{remark}{Remark}

\begin{document}

\title[One-dimensional domain walls]
      {One-dimensional domain walls in thin ferromagnetic films with fourfold anisotropy}

\author{Ross G. Lund and Cyrill B. Muratov}
\address{Department of Mathematical Sciences, NJIT, University Heights, Newark, NJ 07102, USA}
\eads{\mailto{lund@njit.edu}, \mailto{muratov@njit.edu}}

\begin{abstract} 
  We study the properties of domain walls and domain patterns in
  ultrathin epitaxial magnetic films with two orthogonal in-plane
  easy axes, which we call fourfold materials. In these materials, the
  magnetization vector is constrained to lie entirely in the film
  plane and has four preferred directions dictated by the easy
  axes. We prove the existence of $90^\circ$ and $180^\circ$ domain
  walls in these materials as minimizers of a nonlocal one-dimensional
  energy functional.  Further, we investigate numerically the role of
  the considered domain wall solutions for pattern formation in a
  rectangular sample.
\end{abstract}

\vspace{2pc}
\noindent \hspace{-2.5mm} \emph{Keywords}: magnetic domains, thin
films, non-local variational problems

\submitto{Nonlinearity}
\section{Introduction}

Thin-film ferromagnetic materials have played a central role in
information storage technologies for many years \cite{Hubert, Moser,
  eleftheriou10}. In this context, much attention has been
devoted to studying these materials by both the physics and
mathematics communities \cite{Hubert,
  DeSimoneKohnMullerOtto2}. Magnetic storage media make use of
\emph{magnetic domains}---regions of uniform magnetization separated
by thin transition layers called \emph{domain walls}---to represent
bits of information.  Typically, thin films possess one
distinguished in-plane direction along which the magnetization prefers
to align; this direction is referred to as an easy axis, and materials
possessing a single easy axis as \emph{uniaxial}. In such a material,
this results in two distinct optimal domain orientations, which may be
used to store binary information. More recently, more preference has
been given to perpendicular materials --- magnetic storage materials
in which the easy axis is in the out-of-plane direction \cite{Moser,
  PMR}. Nevertheless, thin film materials with in-plane anisotropy
continue to be important for many applications, such as
magnetoresistive random access memory \cite{Dennis, Akerman, Tehrani,
  Zhu08, Roadmap} and spintronics \cite{Bader}.

In this article, we study one-dimensional domain walls in thin
ferromagnetic films in which the magnetization is strongly penalized
from pointing out of the film plane, with two orthogonal in-plane easy
axes (and thus four optimal magnetization directions). We refer to
these materials as \emph{ultrathin fourfold films}. Such behaviour is,
for example, experimentally realized in very thin (3-19 monolayers
thick) films of epitaxial cobalt and gives rise to unusual magnetic
domain morphologies in which the magnetization vector has a tendency
to rotate by integer multiplies of $90^\circ$ in the film plane
\cite{Oepen1, Oepen2}.

In order to understand the magnetization behaviour of thin-film
materials from a theoretical point of view, one would like to study
the formation of domain patterns, and  the structure of the domain
walls which connect  them. We start from the Landau--Lifshitz--Gilbert
(LLG) equation  for the dynamics of the magnetization vector
$\v{M}(\v{x},t)$, which is  of fixed length $|\v{M}(\v{x},t)|=M_s$ and
defined on a spatial  domain $\Omega \subset \R^3$ representing the
ferromagnetic  body under study. The LLG equation for $\v{M}$ reads
\beq
\pd{\v{M}}{t}= -\frac{\gamma}{1+\alpha^2} \brackets{\v{M}\times\v{H} +
  \alpha  \v{M}\times \v{M}\times\v{H}}, \quad \uv{n}\cdot \grad
\v{M}|_{\partial  \Omega} = 0,
\label{LLG}
\eeq
where $\gamma$ is the gyromagnetic ratio, $\alpha$ is the
dimensionless Gilbert  damping parameter, and $\v{H}$ is the effective
magnetic field. This field is obtained via 
$\v{H} = - {\delta E}/{\delta \v{M}}$,
where ${E}(\v{M})$ is the micromagnetic energy functional:
\begin{multline} {E}(\v{M}) = \half \int_\Omega
  \brackets{\frac{A}{M_s^2} \abs{\grad\v{M}}^2 + \frac{K}{M_s^2}
    \Phi(\v{M}) - 2 \v{H}_a\cdot \v{M} }\d^3\v{r} \\+\half \int_{\R^3}
  \int_{\R^3}
  \frac{\div\v{M}(\v{r})\div\v{M}(\v{r}')}{\abs{\v{r}-\v{r}'}}\d^3\v{r}\d^3\v{r}'.
\label{3DEnergy}
\end{multline}
Here, $A$ is the exchange constant; $K$ is a crystalline
anisotropy constant, with $\Phi(\v{M})$ a scalar function describing
the anisotropy; and $\v{H}_a \in \R^3$ is an external applied magnetic
field. The terms in the energy may be understood as follows. The first
term is the exchange energy, which penalizes spatial variations of
$\v{M}$; the second is the anisotropy energy, which describes the
preferred directions for $\v{M}$ within a material; the third is the
Zeeman energy, which prefers $\v{M}$ to align with the external
field; and the fourth (in which $\v{M}$ is extended by zero outside of
$\Omega$ and the derivatives are understood in the distributional
sense) is the nonlocal magnetostatic energy, which prefers to minimize
the distributional divergence of $\v{M}$. See e.g. \cite{Hubert} for
further details of the micromagnetic model.

We can consider observable static domain patterns as stationary (i.e.
with $\pd{\v{M}}{t}=0$) solutions of \eqref{LLG}, and note that, in
this case, equation \eqref{LLG} coincides with the Euler--Lagrange
equation for $E$ incorporating the pointwise constraint
$\abs{\v{M}} = M_s$. This will enable us to employ variational
techniques to characterize solutions.

We consider a reduction of the micromagnetic theory appropriate for
very thin films. Many such reductions have been presented before
\cite{GioiaJames, DeSimoneKohnMullerOtto, MoserR, Kurzke,
  KohnSlastikov}, corresponding to a variety of regimes of the
physical parameters in the energy \eqref{3DEnergy}. In order to
understand the parameter regime we study here, we introduce the
following quantities: \beq \ell = \brackets{\frac{A}{4\pi
    M_s^2}}^{1/2}, \quad L=\brackets{\frac{A}{K}}^{1/2}, \quad Q =
\brackets{\frac{\ell}{L}}^{2},
\label{Parameters}
\eeq respectively called the exchange length, Bloch wall width, and
quality factor. In extended films, we may take the spatial domain
$\Omega = \R^2 \times (0,d)$, where $d$ is the film thickness. The
physical regime we consider is then characterized by the scalings
$d \lesssim \ell \lesssim L$, with $Ld/\ell^2 \sim 1$. This regime
(ultrathin, moderately soft film) is relevant for a variety of
materials \cite{Heinrich}.  In this regime, one may introduce the
dimensionless \emph{thin-film parameter} \beq \nu =\frac{4\pi M_s^2
  d}{KL} = \frac{Ld}{\ell^2} =\frac{d}{\ell \sqrt{Q}}, \eeq which
characterizes the strength of the magnetostatic interaction relative
to both exchange and anisotropy.

One may then formally derive a reduced LLG equation from \eqref{LLG}
by considering the limit $Q\to 0$ and $d\to 0$ together with
$\nu = O(1)$ fixed \cite{GarciaCerveraE, Murosi}. Letting
$\v{m} = \v{M}/M_s$ and assuming $\v{m}=(m_1, m_2, 0)$, i.e., that
  $\v{m}$ lies entirely within the film plane, after suitable
rescalings one finds the effective overdamped equation

\beq \pd{\v{m}}{t} = - \v{m}\times \v{m}\times \v{h}.\label{2DLLG}
\eeq Here, $\v{h}$ is the effective field obtained now as
$\v{h}=-\delta \mathcal{E}/\delta \v{m}$, where $\mathcal{E}$ is
the reduced thin-film energy \beq \mathcal{E}(\v{m}) = \frac12
\int_{\R^2} \brackets{|\grad \v{m}|^2 + \Phi(\v{m})} \d^2 \v{r} +
\frac{\nu}{8\pi}\int_{\R^2}\int_{\R^2}
\frac{\div\v{m}(\v{r})\div\v{m}(\v{r}')}{|\v{r}-\v{r}'|}\d^2 \v{r}
\d^2 \v{r}',
\label{2DEnergy}
\eeq and now $\v{m}:\R^2\to \mathbb S^1$ is the in-plane magnetization
direction field. In what follows, we consider the case of \beq
\Phi(\v{m}) = (\v{m}\cdot \v{e}_1)^2(\v{m}\cdot\v{e}_2)^2, \eeq
corresponding to fourfold anisotropy. This type of magnetocrystalline
anisotropy is very common for ultrathin epitaxial films
\cite{Heinrich}. Ultrathin films with this type of anisotropy has been
proposed for applications to multi-level magnetoresistive random
access memories \cite{moieee,mov15} and could be of interest to domain
wall based devices \cite{Roadmap, Logic}.

In uniaxial thin films, the behaviour of $180^\circ$ domain walls has
been extensively studied. For simple one-dimensional profiles
connecting the two optimal directions, there are two possibilities:
the \emph{Bloch wall}, in which the magnetization transitions between
the optimal domains by rotating out of the film plane, and the
\emph{N\'eel wall}, in which the rotation occurs entirely within the
plane. Which wall type is energetically preferred depends essentially
on how severe the penalty for rotating out of plane is, which in turn
depends on the film thickness. In the ultrathin regime we consider in
this article, this penalty is strong enough to simply forbid any out
of plane component of $\v{m}$, so that the N\'eel wall profile is the
only choice.

N\'eel walls have been studied for many years, with a degree of
controversy (see e.g. \cite{Aharoni, Hubert}). More recent
micromagnetic studies have led to a good present understanding of the
Neel wall's internal structure \cite{Hubert, DeSimoneKohnMullerOtto2,
  Garcia1, Melcher, Garcia2, CapellaOttoMelcher, Chemur, MurYan}, the
main features of which (sharp inner core with slowly decaying tails)
have been verified experimentally \cite{BergerOepen, WongLaughlin,
  Jubert}.

Rigorous mathematical studies of the N\'eel walls began with the work
of \textsc{Garcia-Cervera}, who studied, both analytically and
numerically, the one-dimensional variational problem obtained from the
full micromagnetic energy by restricting to profiles which depend only
on one spatial variable \cite{Garcia1,Garcia2}. The same functional
was studied by \textsc{Melcher}, who restricted the admissible
magnetization configurations to those constrained to the film plane,
and established symmetry and monotonicity of minimizers connecting the
two optimal directions \cite{Melcher}. Uniqueness of the N\'eel wall
profile and its linearized stability with respect to one-dimensional
perturbations was treated by \textsc{Capella, Otto and Melcher}
\cite{CapellaOttoMelcher}. Stability of geometrically constrained
N\'eel walls with respect to large two-dimensional perturbations has
been demonstrated asymptotically in \cite{DeSimoneKnupferOtto}.
Most recently a comprehensive study of N\'eel walls under the
influence of applied magnetic fields was undertaken by
\textsc{Chermisi and Muratov} \cite{Chemur}. They proved existence,
uniqueness, strict monotonicity and smoothness of the wall profile
along with estimates for its asymptotic decay.

To summarize, in ultrathin uniaxial films the magnetization is
effectively constrained to lie completely in the film plane, and one
encounters $180^\circ$ N\'eel walls as the optimal transition layer
profiles connecting the two uniform states. These are now well
understood. Beyond that, it is possible to observe stable
\emph{winding domain walls}, in which the magnetization makes a number
of full $360^\circ$ rotations (most often just one, though more are
possible) \cite{SmithHarte, Wade1, Wade2}. This type of domain walls
has received recent theoretical attention in \cite{Murosi2,
  MuratovKnupfer}. A reproducible way to inject 360$^\circ$ walls
  into ferromagnetic nanowires and successful manipulation of such
  domain walls were recently demonstrated experimentally in
  \cite{ZhangRoss, JingRoss}.

In fourfold films with in-plane magnetizations, we can make the
following analogies with the uniaxial case.  In fourfold materials,
the $90^\circ$-walls are expected to exist as optimal profiles
connecting two adjacent minima of the potential $\Phi$ (e.g.
$+\v{e}_1$ and $+\v{e}_2$). This is analogous to the $180^\circ$
N\'eel walls in uniaxial materials. For a $180^\circ$-wall in a
fourfold material, the magnetization has to connect two nonadjacent
minima of $\Phi$ while passing directly through a third somewhere in
between (i.e. connect $+\v{e}_1$ and $-\v{e}_1$, while passing through
$+\v{e}_2$). Moreover, this should occur without the wall simply
splitting into two separate $90^\circ$-walls. This is analogous to the
$360^\circ$-walls in uniaxial materials.

In this article we extend the methods contained in previous work
concerning $180^\circ$ and $360^\circ$ domain walls in uniaxial
materials to the setting of fourfold materials, and prove existence
results for both $90^\circ$ and $180^\circ$ walls in these materials.
These walls, despite some apparent analogies with those found in
uniaxial films, have not been previously investigated
theoretically.
 
\subsection{Reduced model for one-dimensional domain walls}

Since stationary solutions of \eqref{2DLLG} coincide with critical
points of \eqref{2DEnergy}, in order to study stationary
one-dimensional domain wall profiles, we now seek to derive a 1D
variational problem from \eqref{2DEnergy} which is appropriate to
capture such profiles via minimization.

In what follows we explicitly restrict to stationary profiles,
$\v{m}(x_1, x_2,t)=\v{m}(x_1, x_2)$. It is convenient to
introduce the in-plane magnetization angle $\theta:\R^2 \to \R$ via
\beq \v{m} = - \v{e}_1\sin\theta + \v{e}_2 \cos\theta. \eeq We
now assume a one-dimensional profile
$\theta(x_1, x_2)=\theta(\xi)$ varying only along the direction
$\v{e}_\xi = \v{e}_1 \cos\beta + \v{e}_2 \sin \beta $; we refer to the
angle $\beta$ as the \emph{wall orientation}.
With these assumptions, the LLG equation \eqref{2DLLG} for a
stationary 1D profile  $\theta(x)$ reduces to
\beq
0 = -\theta_{xx} +
\frac14\sin4\theta+\frac{\nu}{2}\cos(\theta-\beta) \halflap
\sin(\theta-\beta),
\label{1DELGeneral}
\eeq where $\halflap$ is the negative 1D half-Laplacian (a linear
operator from $H^1(\R)$, modulo additive constants, to $L^2(\R)$
whose Fourier symbol is $|k|$).  Equation \eqref{1DELGeneral} is also
the Euler--Lagrange equation corresponding to the energy \beq
\mathcal{E}_\beta(\theta) = \frac12 \int_\R \brackets{|\theta'|^2
  +\frac14 \sin^2 2\theta} \d x +
\frac{\nu}{4}\norm{\sin(\theta-\beta)}_{\dot{H}^{1/2}(\R)}^2,
\label{1DEnergyGeneral}
\eeq where we introduced the homogeneous $H^{1/2}(\R)$
(semi-)norm \cite{LiebLoss}: \beq
\norm{u}_{\dot{H}^{1/2}(\R)}^2 =\int_\R u \halflap u \d x =
\frac{1}{2\pi}\int_\R\int_\R \frac{(u(x)-u(y))^2}{(x-y)^2} \d x \d y.
\eeq It is not too difficult to see that this energy corresponds to
the energy $\eqref{2DEnergy}$ of a 1D profile per unit width in the
transverse direction.

This model forms the basis of the rest of this article. It is
necessary to specialize it further to individually examine the two
types of wall we study: $90^\circ$ and $180^\circ$ walls.
To state our main results, we need to introduce the following
  general admissible class of functions for a given $\alpha \in
  \mathbb R$:
$$
\mathcal{A}_{\alpha} = \{\theta \in
H^1_{\text{\scriptsize{loc}}}(\R) : \theta - \eta_{\alpha} \in
H^1(\R)\},
$$
where $\eta_{\alpha}\in C^\infty(\R)$ is a fixed function which
  satisfies
$$
\eta_{\alpha}(x)  =  \begin{cases} \alpha & \text{ for }
  x\in (-\infty,-1), \\ 0 & \text{ for } x\in (1,+\infty).\end{cases}
$$
It is easy to see that the definition of the admissible class
$\mathcal A_\alpha$ does not depend on the specific choice of
$\eta_\alpha$ \cite{Chemur}. Also, by Morrey's theorem we may
  always assume that if $\theta \in \mathcal A_\alpha$, then
  $\theta \in C(\R) \cap L^\infty(\R)$ and that
  $\lim_{x \to +\infty} \theta(x) = 0$ and
  $\lim_{x \to -\infty} \theta(x) = \alpha$. 

In the following, we will look for minimizers of the
one-dimensional domain wall energy $\mathcal E_\beta$ in the
admissible classes $\mathcal A_\alpha$ with $\alpha = \pi/2$ and
$\alpha = \pi$ to study $90^\circ$ and $180^\circ$ walls,
respectively.

\section{Main Results}

We are now in a position to state the main results of this work.  The
first result below concerns $90^\circ$-walls, and provides existence
of these as energy minimizing configurations.  These profiles only
exist for a particular orientation $\beta = -\pi/4$ (modulo $\pi/2$
rotations).  Furthermore, we are able to extract further information
on the profiles including uniqueness, smoothness, and strict
monotonicity.

\begin{thm}[$90^\circ$-walls: existence, uniqueness, regularity and strict monotonicity]\label{90Exist}
  For $\beta=-\pi/4$ and each $\nu>0$, there exists a minimizer of the
  energy $\mathcal{E}_\beta(\theta)$ over the admissible class
  $ \mathcal{A}_{\pi/2}$. The minimizer is unique (up to
  translations), strictly decreasing with range equal to
  $(0,\pi/2)$, and is a smooth solution of \eqref{1DELGeneral}
  satisfying the limit conditions \begin{equation}
  \label{thbcpi2}
  \lim_{x\to+\infty}\theta(x) = 0, \quad
  \lim_{x\to-\infty}\theta(x) = \pi/2.  
\end{equation}
Moreover, if $\thetamin:\R\to (0,\pi/2)$ is the minimizer of
$\mathcal{E}_{-\pi/4}(\theta)$ over $\mathcal{A}_{\pi/2}$
satisfying $\thetamin(0)=\pi/4$, then
$\thetamin(x)=\pi/2 - \thetamin(-x)$.
\label{90WallsResult}
\end{thm}
\noindent The choice of the admissible class $\mathcal A_{\pi/2}$
serves to enforce the asymptotic behavior in \eqref{thbcpi2}. We note
that for other wall orientations $\beta$ the wall would carry a net
line ``charge'' and, hence, the last term in the one-dimensional wall
energy \eqref{1DEnergyGeneral} would always be infinite on
  $\mathcal A_{\pi/2}$.

The next result is similar to the first, but concerning
$180^\circ$-walls. Again the orientation of the walls is restricted,
this time to $\beta = 0$ (modulo $\pi/2$ rotations). Many of the
properties of $90^\circ$ walls follow here, though uniqueness of the
profile is not presently clear.

\begin{thm}[$180^\circ$-walls: existence, regularity and strict
  monotonicity]\label{180Exist}
  For $\beta=0$ and each $\nu>0$, there exists a minimizer of the
  energy $\mathcal{E}_\beta(\theta)$ over the admissible class
  $ \mathcal{A}_{\pi}$. The minimizer is strictly decreasing with
  range equal to $(0,\pi)$, and is a smooth solution of
  \eqref{1DELGeneral} satisfying the limit conditions
  \begin{equation}
  \label{thbcpi}
  \lim_{x\to+\infty}\theta(x) = 0, \quad
  \lim_{x\to-\infty}\theta(x) = \pi.
\end{equation}
Moreover, if $\thetamin:\R\to (0,\pi)$ is the minimizer of
$\mathcal{E}_0(\theta)$ over $\mathcal{A}_{\pi}$ satisfying
$\thetamin(0)=\pi/2$, then $\thetamin(x)=\pi - \thetamin(-x)$.
\label{180WallsResult}
\end{thm}

\noindent Similarly to the case of 90$^\circ$-walls, the choice of
  the admissible class $\mathcal A_\pi$ ensures the conditions
  \eqref{thbcpi} at infinity, and for other choices of wall
  orientation
  there is a net line charge as well.

The remainder of this article is structured as follows. In \S3, we
present proofs of the results given above, using primarily variational
methods. In \S4, we conduct a
numerical study of 1D domain walls in fourfold materials using 1D
simulations, and perform 2D simulations of a rectangular film element
to observe static magnetization configurations involving these walls.
Finally in \S5 we conclude and suggest some further extensions to this
work.



\section{Proofs of main results}

The following section is devoted to
motivating the statements of theorems \ref{90WallsResult} and
\ref{180WallsResult}, and presenting their
proofs.


\subsection{$90^\circ$ walls: Proof of theorem \ref{90WallsResult}}

Let us begin by motivating the precise statement of the theorem.
Firstly we note that in principle, the result one would like to obtain
is existence of a solution to \eqref{1DELGeneral} which satisfies
(without loss of generality) the conditions \beq
\lim_{x\to+\infty}\theta(x) = 0, \quad \lim_{x\to-\infty}\theta(x) =
\pi/2,
\label{90Limits}
\eeq 
and is in some sense physical, i.e. having finite energy per unit
  length of the domain wall. Let us recall
the energy per unit length in \eqref{1DEnergyGeneral} for a 1D
wall of orientation $\beta$. In explicit terms, it reads
\begin{multline}
  \mathcal{E}_\beta(\theta) = \frac12 \int_\R
  \brackets{|\theta'|^2 +\frac14 \sin^2 2\theta} \d x \\
  +\frac{\nu}{8\pi}\int_\R\int_\R
  \frac{(\sin(\theta(x)-\beta)-\sin(\theta(y)-\beta))^2}{(x-y)^2} \d x
  \d y. \end{multline}

We would like to choose an admissible class of minimizers
corresponding to $90^\circ$ transition layers with finite energy which
connect two of the global minima of $\mathcal{E}_\beta$ (given by
$\theta(x) = N \pi/2$ for any $N\in \mathbb{Z}$). Without loss of
generality, we can consider profiles satisfying
\eqref{90Limits}. For
$\theta \in H^1_\text{\scriptsize{loc}}(\R)$ the
local part of the energy is locally well-defined. In order that the
nonlocal term be finite for such profiles, we must constrain the wall
orientation $\beta$ appropriately so as to avoid incurring a net
magnetic charge across the wall. To accomplish this we take
$\beta = -\pi/4$.

The 1D $90^\circ$-wall energy may be expressed as
\beq \mathcal{E}_{-\pi/4}(\theta) = \frac12 \int_\R
\brackets{|\theta'|^2 +\frac14 \sin^2 2\theta} \d x +
\frac{\nu}{4}\norm{\sin(\theta+ {\pi}/{4})}_{\dot{H}^{1/2}(\R)}^2,
\label{90Energy}
\eeq
and the Euler--Lagrange
equation associated to  $\mathcal{E}_{-\pi/4}$ is now given
formally by
\beq
0 = -\theta_{xx} + \frac14 \sin4\theta
+\frac{\nu}{2}\cos(\theta+\pi/{4}) \halflap  \sin(\theta+{\pi}/{4}),
\label{90EL}
\eeq with limit conditions \eqref{90Limits}.

The motivation for the statement of theorem \ref{90WallsResult} should
now be clear. We present a slightly abbreviated proof of this result;
much of the machinery follows directly from the work of
\textsc{Chermisi and Muratov} \cite{Chemur} concerning N\'eel walls in
uniaxial materials. Thus we refer the reader to their work when
proving certain steps, and focus here on aspects which are
significantly different.

It is convenient to first record some preliminary lemmas. 

\begin{lemma}[Restriction of rotations]\label{RestrictionOfRotations}
  Let $\theta \in \mathcal{A}_{\pi/2}$.
  Then there exists
  $\tilde{\theta}\in \mathcal{A}_{\pi/2}$ such that
  $\tilde{\theta}(\R)\subset[0,\pi/2]$ and
  $\mathcal{E}_{-\pi/4}(\tilde{\theta})\leq
  \mathcal{E}_{-\pi/4}({\theta})$.
\end{lemma}
\begin{proof}
We let $\theta^T:\R\to [0,\pi/2]$ be defined, for $k\in \mathbb{Z}$, by
$$
\theta^T(x) =
\begin{cases}
\theta(x) - k\pi & \text{ if } \theta(x) \in [k\pi, (k+\half) \pi),
\\
k\pi-\theta(x)  & \text{ if } \theta(x) \in [(k-\half)\pi, k \pi).
\end{cases}
$$
We then have
$$
\norm{\sin 2\theta}^2_{L^2(\R)} = \norm{\sin 2\theta^T}^2_{L^2(\R)}, \quad \norm{\theta'}^2_{L^2(\R)} = \norm{(\theta^T)'}^2_{L^2(\R)},
$$
and the inequality in the lemma follows from setting $\tilde{\theta}=\theta^T$, the inequality
$$
\norm{u}^2_{\dot{H}^{1/2} }= \int_\R u\halflap u \d x \geq \int_\R |u|\halflap |u| \d x, 
$$
and the fact that $\abs{\sin(\theta+\pi/4)}=\sin(\theta^T+\pi/4)$.
\end{proof}

Given this lemma, we may restrict the admissible class to those
$\theta \in \mathcal{A}_{\pi/2}$ also satisfying
${\theta}(\R)\subset[0,\pi/2]$. It is then useful to define a function
$\rho:\R\to[0,\pi/4] \in H^1(\R)$, corresponding to each $\theta$ in
this restricted class, via \beq \rho(x) = \begin{cases} \theta(x) &
  \text{ if } \theta(x)\in[0,\pi/4), \\ \pi/2 -\theta(x) & \text{ if }
  \theta(x)\in[\pi/4,\pi/2].\end{cases}
\label{rhodef}
\eeq It is clear that $\rho$ satisfies
$\mathcal{E}_{-\pi/4}(\rho) =
\mathcal{E}_{-\pi/4}(\theta)$. One then has the following lemma:
\begin{lemma}[Coercivity]\label{LowerBound} Let $\theta \in
    \mathcal{A}_{\pi/2}$
    be such that $\theta(\mathbb R) \subset [0, \pi/2]$ and let $\rho$
    be defined as in \eqref{rhodef}. Then
$$\mathcal{E}_{-\pi/4}(\rho) \geq
\frac{1}{4}\norm{\rho}^2_{H^1(\R)} +
\frac{\nu}{4}\norm{\sin(\rho+{\pi}/{4})}^2_{\dot{H}^{1/2}(\R)}.$$ 
\begin{proof}
  The bound follows immediately from the fact that, since
  $\rho(x) \in [0,\pi/4]$ for all $x\in \R$, one has
$$
\cos \rho(x) \geq 1/\sqrt{2}, \quad \sin \rho(x) \geq \rho(x)/\sqrt{2},
$$
and the identity $\sin^2 2\rho = 4 \sin^2\rho \cos^2\rho$.
\end{proof}
\end{lemma}

The following lemma provides useful properties for candidate minimizers.
\begin{lemma}[Rearrangement]\label{90rearr}
  Let $\theta\in \mathcal{A}_{\pi/2}$ with
  $\theta(\R)\subset [0,\pi/2]$. Then
  $\exists \theta^o \in C(\R; [0,\pi/2])$ satisfying
  $\mathcal{E}_{-\pi/4}(\theta^o)\leq \mathcal{E}_{-\pi/4}(\theta)$
  and the following properties:
$$
\lim_{x\to +\infty}\theta^o(x)=0,\quad \lim_{x\to
  -\infty}\theta^o(x)=\frac{\pi}{2},\quad
\theta^o(0)=\frac{\pi}{4},\quad
\theta^o(x)=\frac{\pi}{2}-\theta^o(-x), 
$$
and $\theta^o$ is non-increasing.
\label{rearr-norm}
\end{lemma} 
\begin{proof}
  The proof here is similar to that of lemma 4 in \cite{Chemur}. Let
  $\rho:\R\to[0,\pi/4]$ be defined as in \eqref{rhodef} and let
  $\rho^*$ denote its symmetric decreasing rearrangement. By standard
  properties of rearrangements (as in \cite{Chemur}) we have firstly
  that \beq \int_\R \sin^2 2\rho^* \d x = \int_\R \sin^2 2 \rho \d x .
\label{L4-ani}
\eeq Secondly,
$$ [\sin(\rho+\pi/4)-\sin(\pi/4)]^* = \sin(\rho^*+\pi/4)-\sin(\pi/4).
$$
Using this, along with lemma 3 in \cite{Chemur}, we see that \beq
\norm{\sin(\rho+\tfrac{\pi}{4})}^2_{\dot{H}^{1/2}(\R)} \geq
\norm{\sin(\rho^*+\tfrac{\pi}{4})}^2_{\dot{H}^{1/2}(\R)}.
\label{L4-nonloc}
\eeq
From \cite[Lemma 7.17]{LiebLoss}, we have
\beq
\int_\R (\rho^*_x)^2 \d x \leq \int_\R \rho_x^2 \d x.
\label{L4-exc}
\eeq Thus, combining \eqref{L4-ani}, \eqref{L4-nonloc}, and
\eqref{L4-exc}, we obtain
$$
\mathcal{E}_{-\pi/4}(\rho^*) \leq \mathcal{E}_{-\pi/4}(\rho)
= \mathcal{E}_{-\pi/4}(\theta).
$$
Finally, defining $\theta^o \in C(\R; [0,\pi/2])$ via \beq
\theta^o(x) = \begin{cases} \rho^*(x) & \text{ if } x\geq0 \\ \pi/2
  -\rho^*(-x) & \text{ if } x<0,\end{cases}
\label{invdef}
\eeq it is clear that
$\mathcal E_{-\pi/4}(\theta^o)=\mathcal E_{-\pi/4}(\rho^*)$, and
that $\theta^o$ satisfies the properties given in the lemma.
\end{proof}

We now turn to the proof of theorem \ref{90WallsResult}.

\emph{Step 1: Existence.}  Take a minimizing sequence
$\{\theta_n\} \subset \mathcal{A}_{\pi/2}$. By translation-invariance
and lemmas \ref{RestrictionOfRotations} and \ref{90rearr} we may
assume
$$
\theta_n \in C(\R; [0,\pi/2]), \quad \theta_n(0)=\pi/4, \quad
\theta_n(x)=\pi/2-\theta_n(-x),
$$
and $\theta_n$ are non-increasing.  For each $n$, define
$\rho_n:\R\to[0,\pi/4]$ as in \eqref{rhodef}. From lemma
\ref{LowerBound} we have
$$
\frac{1}{4}\norm{\rho_n}^2_{H^1(\R)} +
\frac{\nu}{4}\norm{u_n}^2_{\dot{H}^{1/2}(\R)} \leq C <\infty,
$$
where $u_n = \sin(\rho_n+\pi/4) - \sin (\pi/4)$.  We may then
extract a subsequence (not relabelled) such that the following weak
convergences hold:
$$
\rho_n \weak \rho \; \text{in} \; H^1(\R), \quad u_n \weak u \;
\text{in} \; H^{1/2}(\R),
$$
Moreover, by
  compact embedding of the spaces $H^1(\mathbb \R)$ and
  $H^{1/2}(\mathbb R)$ into $L^2_{loc}(\R)$, we have, upon extraction
  of a further subsequence:
$$
\rho_n \to \rho \; \text{and} \; u_n \to u \; \text{strongly in} \;
L^2_{loc}(\R) \; \text{and a.e.} \;\text{in} \; \R.
$$  
Therefore, we have $u = \sin (\rho+\pi/4) - \sin (\pi/4)$ a.e. in
$\mathbb R$.  Then, using Fatou's lemma applied to the second
term in the definition of the energy and lower semicontinuity of
homogeneous $H^1(\R)$ and $H^{1/2}(\R)$ norms, one obtains weak lower
semicontinuity of $\mathcal{E}_{-\pi/4}(\rho)$ with respect to
the weak convergences considered. Passing to the limit
$n\to \infty$, we thus obtain
$$
\mathcal{E}_{-\pi/4}(\rho) \leq \liminf_{n\to\infty}
\mathcal{E}_{-\pi/4}(\rho_n) =\liminf_{n\to\infty}
\mathcal{E}_{-\pi/4}(\theta_n).
$$
Given such a $\rho$, we construct a function
$\thetamin:\R\to[0,\pi/2]$ as in \eqref{invdef}, that has the
properties
$$
\thetamin(0)=\pi/4, \quad \thetamin(x)=\pi/2-\thetamin(-x),\quad
\lim_{x\to+\infty} \thetamin(x) = 0,
$$
and $\thetamin$ is non-increasing. Since
$\mathcal{E}_{-\pi/4}(\rho) =
\mathcal{E}_{-\pi/4}(\thetamin)$,
we then conclude that $\thetamin$ is a minimizer.

\emph{Step 2: Regularity.} Since $\thetamin$ is a minimizer, it must
also be a weak solution to \eqref{90EL}. That is, $\psi=\thetamin$ is
a weak solution of the equation \beq \psi'' + a(x) \cos 2\psi + b(x)
\cos (\psi+\pi/4)=0,
\label{ELODE}
\eeq where $a(x) = \half \sin 2\thetamin$ and
$b(x) = \frac{\nu}{2} \halflap \sin(\theta+\pi/4)$. It is easy to
see from monotonicity of $\thetamin(x)$ and the limit conditions
\eqref{90Limits} that $a(x) \in L^2(\R)$. Moreover, since \beq
\norm{\halflap u}_{L^2(\R)}^2 = \norm{d u \over dx}^2_{L^2(\R)} \qquad
\forall u \in H^1(\R),
\label{nonlocalest}
\eeq and since $[\sin(\thetamin+\pi/4)]' \in L^2(\R)$, we may see
that $b(x)\in L^2(\R)$, as well. Examining \eqref{ELODE}, we may then
conclude that $\thetamin_{xx} \in L^2(\R)$, and thus that
$\thetamin_x \in H^1(\R)$. Furthermore, Morrey's theorem then tells us
that $\thetamin_x \in C(\R)\cap L^\infty(\R)$, and so
$\thetamin \in C^1(\R)$. Going a step further, we differentiate $a(x)$
and $b(x)$ to get $a'(x)=\sin 2\thetamin\cos 2 \thetamin \thetamin_x$
and $b'(x) = \frac{\nu}{2} \halflap[(\sin(\thetamin+\pi/4))_x]$.
Clearly $a'(x)\in L^2(\R)$, so $a\in H^1(\R)$.

We would like to show the same for $b(x)$. Applying
$\eqref{nonlocalest}$ to $u= (\sin(\thetamin+\pi/4))_x$, this then
amounts to having to show that $u_x \in L^2(\R)$ (we already know that
$u\in L^2(\R)$). We have
$u_x=(\sin(\thetamin+\pi/4))_{xx} = \thetamin_{xx}
\cos(\thetamin+\pi/4) - (\thetamin_x)^2 \sin(\thetamin+\pi/4)$.
Applying the interpolation inequality as in
\cite{CapellaOttoMelcher}:
$$
\norm{\thetamin_x}_{L^4(\R)} \leq 
\norm{\thetamin_x}_{L^\infty(\R)}^{1/2}
\norm{\thetamin_x}_{L^2(\R)}^{1/2} < \infty,
$$
 one may then conclude that $b'(x)\in L^2(\R)$. Again,
from \eqref{ELODE}, this gives $\thetamin_{xx}\in H^1(\R)$ and thus
$\thetamin_{xx}\in C(\R)\cap L^\infty(\R)$, implying that
$\thetamin\in C^2(\R)$ is a classical solution to
\eqref{90EL}. These arguments may then be bootstrapped to conclude
that $\thetamin \in C^\infty(\R)$.

\emph{Step 3: Strict monotonicity.} The proof here is similar to
that in \cite{Chemur}. From the previous steps we know that
$\thetamin_x \leq 0$ on $\R$ and $\thetamin \in C^\infty(\R)$. We
claim that in fact $\thetamin_x < 0$ on $\R$. First, as in
\cite{CapellaOttoMelcher}, we show $\thetamin_x(0)<0$: if we assume
$\thetamin(0)=\pi/4$ and $\thetamin_x(0)=0$, then uniqueness of
solutions for the initial value problem implies $\thetamin(x)=\pi/4$
identically, contradicting the limit conditions \eqref{90Limits}.

Next, we show $\thetamin_x<0$ on $\R^+$. Assume there exists an
$x_*>0$ such that this is false: i.e.~$\thetamin_x(x_*)=0$. This,
together with the non-increasing property of $\thetamin(x)$ we
already have, implies that $\thetamin_{xx}(x_*)=0$ as well.
Differentiating the Euler--Lagrange equation and evaluating at $x_*$,
we then have
$$
\thetamin_{xxx}(x_*) = \frac{\nu}{2}\cos(\thetamin(x_*) +\pi/4)g(x_*),
$$
where we defined
$$
g(x) = \halflap \cos(\thetamin(x)+\pi/4).
$$
Computing $g(x_*)$ directly, using the integral representation of
$\halflap$ and noticing that by the assumption on $x_*$ the
  obtained integral converges, we find
$$
g(x) = -\frac{1}{\pi}\int_0^\infty \frac{4x_*y \thetamin_x(y)\cos(\thetamin(y)+\pi/4)}{(x_*-y)^2(x_*+y)^2}\d y.
$$
Since $\thetamin_x \leq 0$ on $\R$, with the inequality strict on a set of positive measure, we see that $g(x_*)>0$. Moreover, we know that $\cos(\thetamin+\pi/4)>0$ on $\R^+$, and hence $\thetamin_{xxx}(x_*)>0$, implying that $\thetamin$ is locally increasing at $x_*$. This is a contradiction. Strict monotonicity on $\R$ follows from the fact that $\thetamin(x) = \pi/2 -\thetamin(-x)$ as proved in step 1.

\emph{Step 4: Uniqueness (up to translations).}  Let $\theta^{(1)}$
and $\theta^{(2)}$ be two distinct minimizers. After suitable
translations these satisfy
$\theta^{(1)}(0)=\theta^{(2)}(0)=\pi/4$. Define
$u =\sin(\theta+\pi/4)$. We may write the energy
$\mathcal{E}_{-\pi/4}$ in terms of $u$ as
$\mathcal E_{-\pi/4}(\theta) = E(u)$, where
$$
E(u) = \frac12 \int_\R \frac{u_x^2}{1-u^2} \d x + \frac12 \int_\R
\brackets{u^2-\frac12}^2 \d x + \frac{\nu}{4} \int_\R u\halflap u \d x
$$
Then define 
$$
\tilde{\theta}(x) = \begin{cases} \sin^{-1}(\half(u^{(1)}(x) +
  u^{(2)}(x))) - \pi/4 & \text{ if } x \geq 0, \\ 3\pi/4 -
  \sin^{-1}(\half(u^{(1)}(x) + u^{(2)}(x))) & \text{ if } x
  < 0,\end{cases}
$$
where $u^{(1)} = \sin(\theta^{(1)}+\pi/4)$ and
$u^{(2)}=\sin(\theta^{(2)}+\pi/4)$. Note that
$\sin(\tilde{\theta}+\pi/4) = \half(u^{(1)} + u^{(2)})$, and is
symmetric decreasing.

Arguing as in \cite{Chemur}, we have
$$
(\tilde{\theta}_x)^2 \leq \frac{(\theta^{(1)}_x)^2+(\theta^{(2)}_x)^2}{2}.
$$
Then, since the last two terms in $E(u)$ above are strictly
convex in $u$ (at least, for $u\in[\frac{1}{\sqrt{2}},1]$, which is
equal to the possible range of values of $u$ that we have here), we
have that
$$
\mathcal{E}_{-\pi/4}(\tilde{\theta})< \frac{\mathcal{E}_{-\pi/4}(\theta^{(1)}) + \mathcal{E}_{-\pi/4}(\theta^{(2)})}{2},
$$
which contradicts the minimality of $\theta^{(1)}$ and $\theta^{(2)}$.
This concludes the proof of Theorem \ref{90Exist}.

\begin{remark}
  It should be possible to use the arguments of \cite{MurYan} to prove
  that the minimizers in theorem \ref{90Exist} are the unique monotone
  $90^\circ$ wall profiles, i.e., that the minimizers are the unique,
  up to translations, monotone critical points of the energy
  $\mathcal E_{-\pi/4}$ in $\mathcal A_{\pi/2}$.
\end{remark}


\subsection{$180^\circ$ walls: Proof of theorem \ref{180WallsResult}}
Similarly to the above analysis of $90^\circ$-walls, we would first
like to write down a 1D wall energy for $180^\circ$-walls and define
an appropriate admissible class for minimizers. We assume,
without loss of generality, a profile connecting the optimal uniform
states $\theta=\pi$ at $-\infty$ and $\theta=0$ at $+\infty$. The
appropriate wall orientation to avoid net charge is given by
$\beta = 0$. The 1D wall energy is thus expressed as \beq
\mathcal{E}_0(\theta) = \frac12 \int_\R \brackets{|\theta'|^2
  +\frac14 \sin^2 2\theta} \d x +
\frac{\nu}{4}\norm{\sin\theta}_{\dot{H}^{1/2}(\R)}^2,
\label{180Energy}
\eeq with the admissible class for minimizers given by
$\mathcal{A}_{\pi}$. The Euler--Lagrange equation associated to
$\mathcal{E}_0$ is given by \beq 0 = -\theta_{xx} + \frac14\sin4\theta
+\frac{\nu}{2}\cos\theta \halflap \sin\theta,
\label{180EL}
\eeq
with limit conditions
\beq
\lim_{x\to+\infty}\theta(x) = 0, \quad
\lim_{x\to-\infty}\theta(x) = \pi.
\label{180BC}
\eeq

We now turn to the proof of theorem \ref{180WallsResult}. Much of the
proof follows by direct analogy with the previous section. Indeed,
following the proof of theorem \ref{90WallsResult}, it is easy to see
that lemmas \ref{RestrictionOfRotations} and \ref{90rearr} generalize
trivially such that we can immediately restrict the admissible class
to non-increasing $\theta\in\mathcal{A}_\pi$ satisfying
$\theta(\R)=(0,\pi)$ along with the properties 
\beq
\quad \theta(0)=\frac{\pi}{2},\quad
\theta(x)=\pi-\theta(-x). \eeq Note that we do not prove uniqueness
here; the methods used to prove uniqueness for $90^\circ$-walls in the
previous section do not apply due to the fact that the anisotropy
energy is nonconvex as a function of $u=\sin\theta$ for some values of
$\theta$ which, in this case, the profile must take. The question of
uniqueness of minimizers for this problem remains open.

Lemma \ref{LowerBound}, however, does not generalize to this
case. This is again due to the same issue of nonconvexity that causes
problems for uniqueness. In order to obtain compactness in
$H^1(\R)$ for the minimizing sequence, we have to prove a bound
on the $L^2(\R)$ norm of $\rho$, this time defined as \beq
\rho(x) = \begin{cases} \theta(x) & \text{ if } \theta(x)\in[0,\pi/2),
  \\ \pi -\theta(x) & \text{ if } \theta(x)\in[\pi/2,\pi]
  , \end{cases}
\label{rhodef180}
\eeq such that once again
$\mathcal{E}_0(\theta)=\mathcal{E}_0(\rho)$. Physically, the
issue with compactness which has occurred here can be interpreted as
the question of whether it is energetically preferable for the
$180^\circ$ transition layer to split into two well-separated
$90^\circ$-walls (which would result in the minimizing sequence weakly
converging to $\pi/2$, which is clearly outside of the admissible
class) or whether the full transition occurs mostly over a finite
interval.

It is clear that the local part of the energy \eqref{180Energy} is
unchanged by having an arbitrarily large region with $\theta=\pi/2$. Below, we show that this is not possible due to the nonlocal term (a similar argument was used
in the analysis of existence of $360^\circ$ walls in uniaxial
materials \cite{MuratovKnupfer}).
Indeed, for $\theta$ in the above class, there exist two numbers $0 < a < b$ such that $\theta(a) = \pi/3$ and
$\theta(b) = \pi/6$. From the anisotropy term alone, we get that
$b - a$ remains bounded above by a multiple of the energy. Using the known symmetry of $\theta$ (i.e. that
$\sin\theta(-x)=\sin\theta(x)$), the nonlocal term in the energy
\eqref{180Energy} is proportional to \beq I=\int_0^\infty
\int_0^\infty
\brackets{\frac{(\sin\theta(x)-\sin\theta(y))^2}{(x-y)^2} +
  \frac{(\sin\theta(x)-\sin\theta(y))^2}{(x+y)^2}}\d x\d y.  \eeq We
can estimate $I$ from below by neglecting the interval $(a,b)$ from
the integrals. Defining \beq f(x,y) =
\frac{(\sin\theta(x)-\sin\theta(y))^2}{(x-y)^2} +
\frac{(\sin\theta(x)-\sin\theta(y))^2}{(x+y)^2} \geq 0, \eeq
 we have \beq I \geq \int_0^a\int_0^a
f(x,y) \d x\d y+ 2\int_0^a\int_b^\infty f(x,y) \d x\d
y+\int_b^\infty\int_b^\infty f(x,y) \d x\d y.  \eeq We may then
estimate the cross term as follows:
$$ \int_0^a\int_b^\infty f(x,y) \d x\d y \geq C\int_0^a\int_b^\infty
\brackets{\frac{1}{(x-y)^2} + \frac{1}{(x+y)^2}}\d x\d y = C
\ln\brackets{\frac{b+a}{b-a}},
$$
for some universal $C > 0$.

One can see that the nonlocal term forces both $a$ and $b$ to be bounded by a multiple of the energy. 
Then, in order to get a bound on  $\norm{\rho}_{L^2(\R)}$, we can use the following bounds on the anisotropy energy
\beq
C \geq \mathcal{E}_0(\rho) \geq\frac14 \int_0^\infty \sin^2 2\theta\d x = \frac14\int_0^b \sin^2 2\theta + \frac14\int_b^\infty \sin^2 2\theta\d x.
\eeq
For the second term, we have $\theta(x) \in [0,\pi/6]$ for $x \in [b,\infty)$, such that $\cos \theta \geq \sqrt{3}/2$ and $\sin \theta \geq \theta\sqrt{3}/2$. Using the identity $\sin^2 2\theta = 4 \cos^2\theta \sin^2\theta$ we then have
\beq
\frac14\int_b^\infty \sin^2 2\theta \geq \frac{9}{16}\int_b^\infty \theta^2 \d x.
\label{BtoInfty}
\eeq
Additionally, 
\beq
\int_0^b \theta^2 \d x \leq C.
\label{ZerotoB}
\eeq
since we know $b$ is finite and $\theta(\R) \subset [0,\pi]$. 
We can then conclude that the transition from $\theta=\pi/2$ to $\theta=0$ as $x$ increases takes place over a finite distance from $x=0$, and that
\beq \norm{\rho}_{L^2(\R)} \leq C, \eeq for some $C > 0$ depending on
$\mathcal E_0(\theta)$.

One then has, for a minimizing sequence $(\rho_n)$, that
$\norm{\rho_n}_{H^1(\R)} \leq C$ and $\norm{\sin
  \rho_n}_{\dot{H}^{1/2}(\R)} \leq C$,
for some $C > 0$ and all $n \in \mathbb N$. One can then extract
weakly convergent subsequences, and the proof of existence of
minimizers, along with the regularity and strict monotonicity, follow
in precisely the same fashion as in the proof of theorem
\ref{90WallsResult}.



\section{Numerical study}

In order to assess the role played by the one-dimensional domain
wall solutions constructed in the preceding sections for the
  domain patterns in two-dimensional films, we use a
finite-difference scheme to solve the overdamped
Landau--Lifshitz--Gilbert equation \eqref{2DLLG}, coupled with
optimal-grid-based methods to compute the stray field. The
algorithm is fully discussed in the work of \textsc{Muratov and
  Osipov} \cite{Murosi}.

\subsection{One-dimensional simulations}
We first aim to solve \eqref{2DLLG} in 1D to obtain the wall profiles
corresponding to our existence results in the previous section. We
evolve the equation above, beginning from initial conditions which
approximate the domain wall profile in question, until a steady state
is reached. Figure \ref{90wallsfig} below displays $90^\circ$
wall profiles for $\nu = 1,5$ and $50$ (upper panels), and their corresponding tails in log-log coordinates (lower panels). Figure \ref{180wallsfig} displays the analogous plots for $180^\circ$
walls. 

Examining figures \ref{90wallsfig} and \ref{180wallsfig}, one can observe that the $90^\circ$ and $180^\circ$ wall profiles behave in much the same manner as $\nu$ is increased. Additionally, the tails all display algebraic decay proportional to $1/x^2$ far from the core, as was proved to be the case for $180^\circ$ N\'eel walls in uniaxial materials \cite{Chemur}. This decay sets in further away from the core as $\nu$ is increased, and is clearly preceded for large $\nu$ (e.g. in the plots for $\nu=50$) by a logarithmic crossover region between the core and the algebraic tails, as is well known for N\'eel walls in uniaxial materials \cite{Hubert, Garcia2, Melcher}. We can conclude that these wall profiles show effectively the same behaviour as N\'eel walls in uniaxial materials. Finally, in figure \ref{180wallsfig}, for $\nu=1$ one can see that the $180^\circ$ wall is starting to separate into two $90^\circ$ walls. This is due to the fact that at $\nu=0$ only the $90^\circ$ walls exist, while for $\nu > 0$ the $180^\circ$ wall is stabilized purely by the magnetostatic interaction.

\begin{figure}[h!]
\begin{center}
\includegraphics[width=0.679\textwidth]{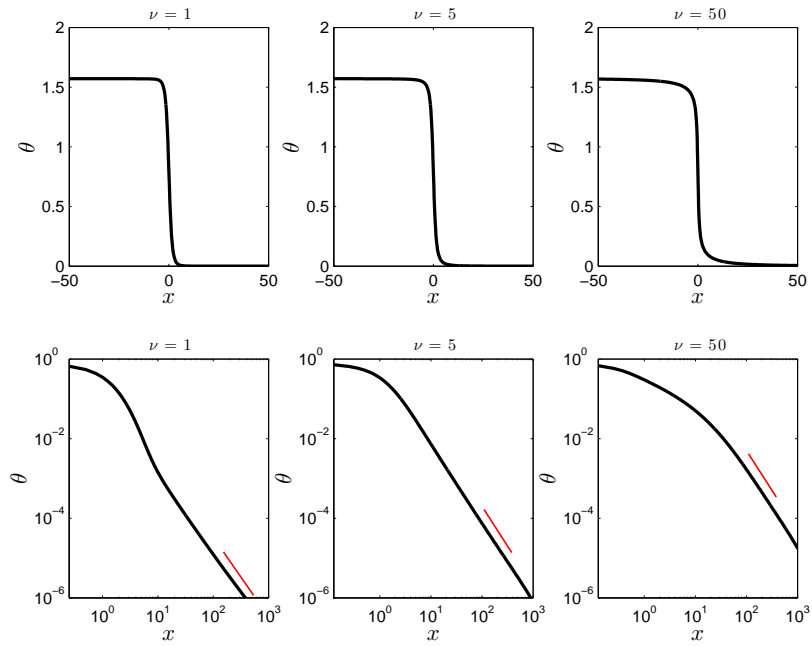}
\caption{Computed 1D 90$^\circ$ wall profiles for $\nu = 1, 5$ and $50$. Upper panels show the wall profiles near the transition layer. Lower panels show the corresponding decay in the tails, plotted in log-log coordinates. Red line segments in the lower panels indicate an algebraic decay of $\theta \sim 1/x^2$.}
\label{90wallsfig}
\end{center}
\end{figure}

\begin{figure}[h!]
\begin{center}
\includegraphics[width=0.679\textwidth]{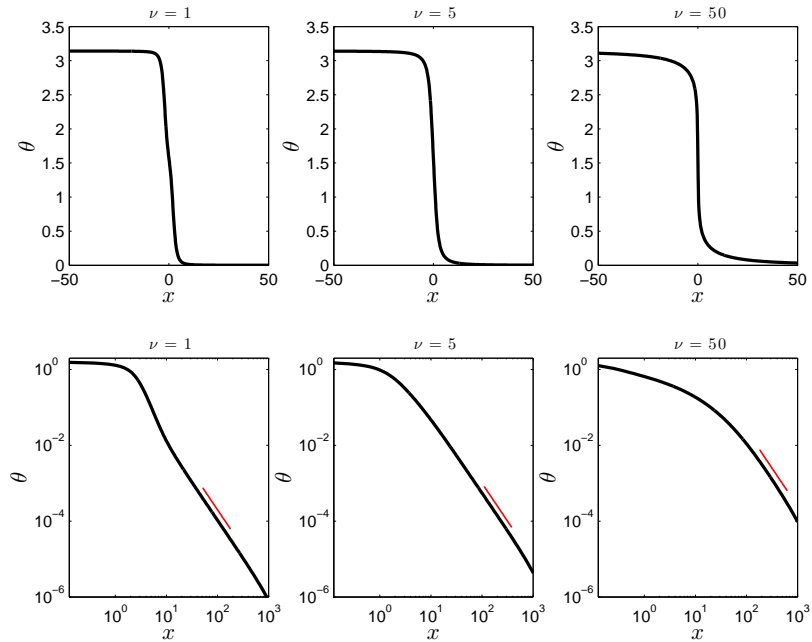}
\caption{Computed 1D 180$^\circ$ wall profiles for $\nu = 1, 5$ and $50$. Upper panels show the wall profiles near the transition layer. Lower panels show the corresponding decay in the tails, plotted in log-log coordinates. Red line segments in the lower panels indicate an algebraic decay of $\theta \sim 1/x^2$.}
\label{180wallsfig}
\end{center}
\end{figure}

\subsection{Two-dimensional simulations}

We now solve \eqref{2DLLG} on a spatial domain
$\Omega =[0, L_x] \times [0, L_y]$ with the edges of the domain
aligned with the easy axes of the material. We have effectively just
three parameters in the model: $L_x, L_y$ and
$\nu$. Each of these has the same intuitive effect on the energy:
increasing domain size or $\nu$ increases the strength of the
magnetostatic interaction, relative to anisotropy and exchange.

In the figures below we display stationary configurations
(remanent states) for a range of domain sizes and values of
$\nu$. Physically, the fixed parameters we use correspond to those of
an epitaxial cobalt film, where $\nu$ can be thought
of as the suitably rescaled film thickness. The parameters (recall the definitions in
equation \eqref{Parameters}) are given approximately by
$\ell \approx 3.37$nm, $Q\approx 0.08$, and thus
$L\approx 12$nm and $\nu$ represents the film thickness in
  nanometers \cite{moieee}.

Starting from 3 different initial states, we can isolate 4 distinct
stationary solutions, as follows. In figure \ref{nusmall} below we
observe the well known `C' (panel (a)) and `S' (panel (b)) states.
These states are close to monodomain states, but with edge domains
appearing along the short edges of the sample to appease the
magnetostatic energy at the boundary. This energy term prefers the
magnetization to align tangent to the boundary wherever possible.
These states result in magnetically charged regions close to the short
edges, but with zero energy density in the bulk of the sample.  There
are necessarily half boundary vortices, which also carry charge, in
two of the corners. The C state has lower energy than the S
state. In both panels, the
size of the sample is given (in units of $L$) by $L_x = 8, L_y=16$
(the figures are to scale), with $\nu=5$. In panel (b), the initial
condition for the simulation was a monodomain with $\theta=\pi/3$; in
panel (a), we took $\theta = \pi/2$ in the lower half of the sample
and $\theta = -\pi/2$ in the upper half. In the figures, the color
indicates the value of $\theta(x,y)$, and the vectors show the
corresponding magnetization field.

\begin{figure}[h!]
\begin{center}
\includegraphics[width=.8\textwidth]{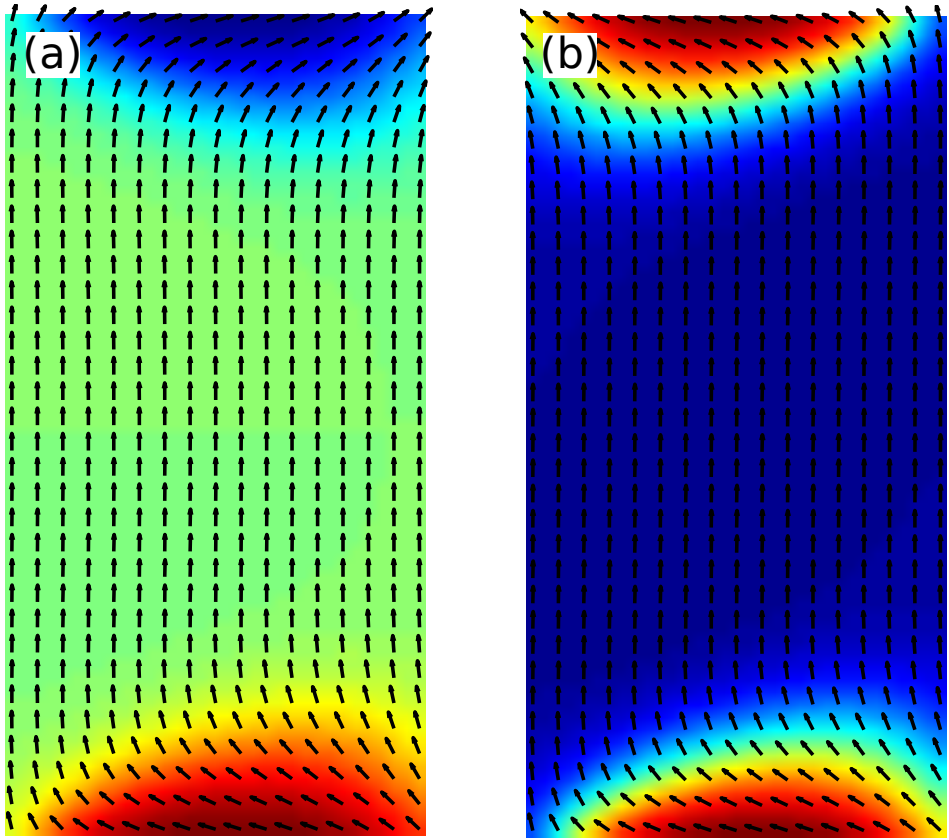}
\caption{(color online) C-state (panel (a)) and S-state (panel (b)). Domain size in both panels is $L_x = 8, L_y=16$, with $\nu=5$.}
\label{nusmall}
\end{center}
\end{figure}

In figure \ref{nularge}, we see instead states consisting of an arrangement of $90^\circ$ and $180^\circ$ domain walls, with interesting phenomena appearing at the lower boundary. In both figures, one has $L_x = 32, L_y=64$, and the simulations were initialized with $\theta=0$ in the right half of the sample and $\theta=\pi$ in the left half. In panel (a), $\nu=5$; in panel (b), $\nu=10$.

In the top halves of both panels, one observes the same phenomenology. There is a $180^\circ$-wall aligned vertically in the center of sample, which splits into two $90^\circ$-walls, which then terminate at the top corners. This arrangement separates the top half of the sample into three domains with orientations along the easy directions $\theta=0,\pi/2$, and $\pi$, which also coincide with the orientations of the edges.  The arrangements are strongly reminiscent of the so-called Landau states \cite{Hubert}; we refer to them as half-Landau states.
\begin{figure}[h!]
\begin{center}
\includegraphics[width=.8\textwidth]{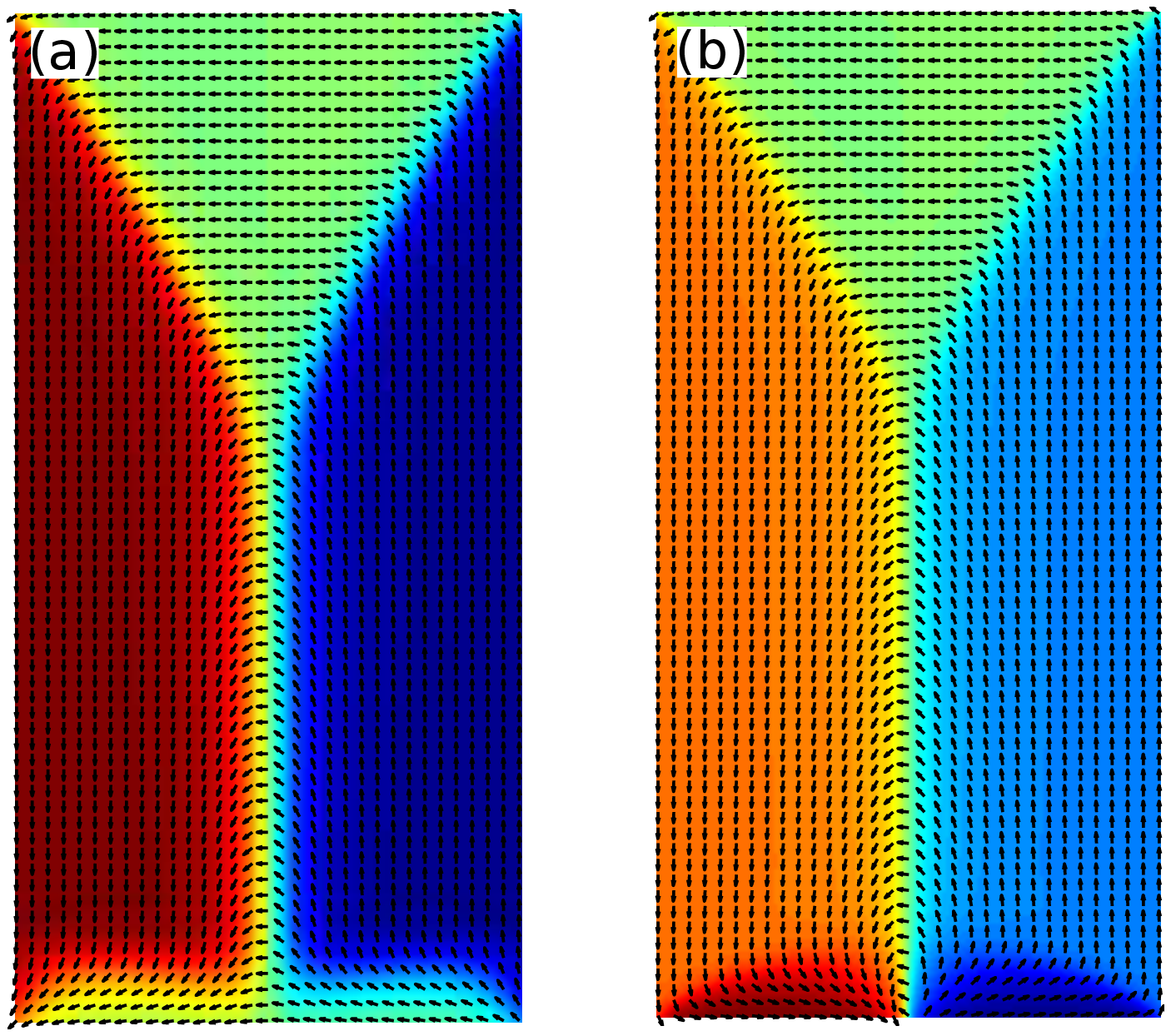}
\caption{(color online) Half-Landau states. Domain size in both panels is $L_x = 32, L_y=64$. In panel (a), $\nu=5$ and the configuration has boundary vortices in the lower corners. In panel (b), $\nu=10$ and we see a configuration with 2 boundary vortices as a bound pair in the center of the lower boundary.}
\label{nularge}
\end{center}
\end{figure}

It is simple to understand why the transition from an edge-domain
state to a half-Landau state is preferred as the relative strength of
the magnetostatic interaction increases, just by considering the top
halves of the figures. The half-Landau states are in principle charge
free (in the top half), while the edge-domains are not. Thus, while
sacrificing some exchange energy in order to form the domain walls,
the magnetostatic energy is reduced.

In principle as mentioned above, according to the theory of section 3,
each of the domain walls in the half-Landau state should be
individually charge-free. While this is true of the $180^\circ$-wall
in the center, it is not quite true of the $90^\circ$ walls
here. Indeed, their orientations are not quite at $45^\circ$ to the
easy axes, and so they are each slightly charged (in opposite
senses). Since the stray-field energy of such walls would diverge
logarithmically as the size of the domain increased, we would expect
that in the large-domain limit, they would converge to a $45^\circ$
orientation.

Let us now discuss the lower edges of the two half-Landau states
pictured in figure \ref{nularge}. Firstly, we note that in the model
we consider, the magnetization is prevented from forming the full
Landau state, since in this state there would necessarily be a
magnetic vortex included in the sample, and these are of infinite
energy in our 2D model. Indeed, admissible magnetization states in
this model must have topological degree zero (i.e. be continuously
deformable to the uniform state); a vortex has degree one, as does a
Landau state.

In panel (a) of figure \ref{nularge}, we have $\nu=5$. On the lower boundary, the magnetization is aligned with $\theta=\pi/2$---the same as the top boundary. There are boundary vortices (quarter vortices) of opposite charges in the lower corners, and what appear to be two oppositely charged $90^\circ$ ``boundary domain walls'' joining the bulk domains on the left and right to the boundary orientation. In panel (b), the magnetization on the lower boundary is instead aligned mostly with $\theta = -\pi/2$, antiparallel to the magnetization at top boundary, and performs a full $2\pi$ rotation in the center of the lower edge (i.e. there is a bound pair or boundary vortices there), such that the whole configuration has degree zero. Additionally, the transition layers at the lower boundary look closer in structure to edge domains than to 1D boundary walls.

In panel (a), the charge on the lower boundary is spread out across the whole edge while the exchange energy is confined close to the edge in the boundary walls. Conversely in panel (b), the charge on the boundary is focused at the bound pair of boundary vortices in the center, while the exchange energy is more spread out across the edge domains. For larger $\nu$, it thus appears energetically preferable to concentrate the charge in a smaller region.


\section{Conclusions and further work}

We have proven existence results concerning $90^\circ$ and $180^\circ$
domain walls, viewed as minimizers of the 1D domain wall energy, in
thin film ferromagnets with fourfold in-plane
anisotropy. Moreover we are able to learn a lot of information about
these structures, including strict monotonicity, smoothness, and, in
the case of $90^\circ$-walls, as well as uniqueness. Further
problems here include proving uniqueness (or not) of the
$180^\circ$-wall, and studying walls with more winding such as
$360^\circ$-walls, or materials with more exotic crystalline
anisotropy (see, e.g., \cite{Ognev}).  We also presented
numerically computed 1D wall profiles corresponding to our existence
results, and magnetization configurations for rectangular samples of
fourfold materials, which feature slightly charged walls in the bulk
and charged walls at the boundary. This investigation poses possible
questions about domain walls in regimes where the penalty for walls
having net charges is relaxed such that charged walls, both in the
bulk and at the boundary may be allowed in the large domain limit, and
it would be interesting to study the $\Gamma$-limit of such a 2D
thin film energy. Additionally, there is the question of existence of
`boundary walls' in the thin-film regime considered in this article:
this will be addressed elsewhere \cite{MuratovLund}.

\section*{Acknowledgements}

This work was supported by NSF via grant DMS-1313687.

~\\

\end{document}